\documentclass[10pt]{article}

\usepackage{amsmath,amsfonts,amssymb,amsthm,graphics,amscd}
\usepackage{latexsym,mathrsfs}

\usepackage{color}

\newtheorem{theorem}{Theorem}[section]
\newtheorem{lemma}{Lemma}[section]
\newtheorem{corollary}{Corollary}[section]

\newtheorem{example}[theorem]{Example}
\theoremstyle{definition}
\newtheorem{definition}{Definition}[section]

\theoremstyle{remark}

\def\ds{\displaystyle}

\def\p{\prime}
\def\inter{{\rm int}}

\def\NN{\mathbb{N}}
\def\CC{\mathbb{C}}

\def\DD{\mathbb{D}}
\def\DDD{\bf{D}}

\numberwithin{equation}{section}

\def\acc{{\rm acc}}
\def\iso{{\rm iso}}
\def\dim{{\rm dim}}

\def\ind{{\rm ind}}
\def\snoi{\smallskip\noindent}

\def\codim{{\rm codim}\, }
\def\NN{{\mathbb N}}

\def\CC{{\mathbb C}}

\def\DD{\mathbb{D}}
\def\W{\mathcal{W}}
\def\B{\mathcal{B}}
\def\M{\mathcal{M}}

\def\Q{\mathcal{Q}}
\def\D{\mathcal{D}}
\def\S{\mathcal{S}}
\def\J{\mathcal{J}}
\def\c{\'c}

\begin{document}

\date{}

\title{Generalized Drazin-meromorphic invertible operators and generalized Kato-meromorphic decomposition}
\author{S. \v C. \v Zivkovi\'c-Zlatanovi\'c, B. P. Duggal\footnote{The authors are
supported by the Ministry of Education, Science and Technological
Development, Republic of Serbia, grant no. 174007.}}

\date{}
\maketitle

\setcounter{page}{1}

\begin{abstract}
A bounded linear operator $T$ on a
Banach space $X$
  is said to be    generalized Drazin-meromorphic
invertible if there exists a bounded linear operator $S $ acting on $X$ such that
$TS=ST$, $STS=S$, $ TST-T$ is meromorphic. We shall say that  $T$  admits a generalized  Kato-meromorphic
decomposition if there exists a pair of $T$-invariant closed subspaces $(M,N)$ such that $X=M\oplus N$, the reduction  $T_M$ is Kato and the reduction $T_N$ is meromorphic.
In this paper we shall investigate such kind of operators and corresponding spectra, the generalized Drazin-meromorphic spectrum and the generalized Kato-meromorphic spectrum, and prove that these spectra are empty if and only if the operator $T$ is polynomially meromorphic. Also we obtain that  the generalized Kato-meromorphic spectrum differs from the Kato type spectrum                              on at most countably many points.
  Among others, bounded linear operators  which can be expressed as a direct sum of a meromorphic operator and a bounded below (resp.  surjective, upper (lower) semi-Fredholm, Fredholm,  upper (lower) semi-Weyl, Weyl) operator are studied. In particular, we shall characterize the single-valued extension property at a point $\lambda_0\in\CC$ in the case that $\lambda_0-T$ admits a generalized  Kato-meromorphic
decomposition. As a consequence we get several results on cluster points of some distinguished parts of the spectrum.

\end{abstract}



\section{Introduction  and preliminaries}

Throughout this paper $\mathbb{N} \, (\mathbb{N}_0)$ denotes the set of all positive
(non-negative) integers, $\mathbb{C}$ denotes the set of all
complex numbers and $L(X)$  denotes  the Banach algebra of bounded linear operators acting on  an infinite dimensional complex  Banach space $X$.
  The group of all invertible operators is
denoted by $L(X)^{-1}$, while  set of all
bounded below (surjective) operators is denoted by  $\J(X)$ ($\S(X)$).
For $T\in L(X)$ we shall denote by $\sigma(T)$, $\sigma_{ap}(T)$ and $\sigma_{su}(T)$ its  spectrum, approximate point spectrum and surjective spectrum, respectively.  Also, we shall write $ \alpha (T)$ for the dimension of the kernel   $N(T)$ and   $\beta (T)$ for the codimension of the range $R(T)$.
We call  $T\in L(X)$ an {\it upper semi-Fredholm} operator if $ \alpha (T)<\infty\text{
 and
}R(T)\text{ is closed}$,
but if $\beta (T)<\infty$, then $T$ is
a {\it lower semi-Fredholm} operator. We use   $\Phi_{+}(X)$ (resp.
$\Phi_{-}(X)$)  to  denote the set of upper (resp. lower)
semi-Fredholm operators.  An operator  $T\in L(X)$ is said to be  a {\it semi-Fredholm} operator if $T$ is upper or lower semi-Fredholm. 
              If $T$ is semi-Fredholm, the index of $T$, $\ind (T)$, is defined  to be
           $\ind (T)=\alpha (T)-\beta (T)$. The set of Fredholm operators is defined as
           $
\Phi (X)=  \Phi_+(X) \cap \Phi_-(X).$
The sets  of upper semi-Weyl, lower semi-Weyl and  Weyl operators are
defined as $\W_+(X)=\{ T\in\Phi_+ (X): \ind (T)\le 0\}$, $\W_-(X)=\{ T\in\Phi_- (X): \ind (T)\ge 0\}$  and  $\W(X)=\{ T\in\Phi (X): \ind (T)=0\}$, respectively.
An operator $T \in L(X)$ is said to be {\em Riesz operator}, if $T-\lambda
\in \Phi(X)$ for every non-zero $\lambda \in \mathbb{C}$.
An operator $T\in L(X)$ is  {\it meromorphic} if its non-zero spectral points are poles of its resolvent, and in that case we shall write $T\in(\M)$. It is well known that $T$ is Riesz operator if and only if
 every nonzero point of  $\sigma (T)$ is a pole of the finite
algebraic multiplicity. So, every Riesz operator is meromorphic.
 We say that $T$ is {\it polinomially meromorphic} if there exists non-trivial polynomial $p$ such that $p(T)$ is meromorphic.

For $n\in\NN_0$ we set $c_n(T)=\dim R(T^n)/R(T^{n+1})$ and $c_n^\prime(T)=\dim N(T^{n+1})/N(T^n)$. From \cite[Lemmas 3.1 and 3.2]{Kaashoek} it follows that $c_n(T)=\codim (R(T)+N(T^n))$ and $c_n^\prime(T)=\dim (N(T)\cap R(T^n))$. Obviously, the sequences $(c_n(T))_n$ and $(c_n^\prime(T))_n$ are decreasing. For each $n\in\NN_0$, $T$ induces a linear transformation from the vector space $R(T^n)/R(T^{n+1})$ to the space $R(T^{n+1})/R(T^{n+2})$: let $k_n(T)$ denote the dimension of the null space of the induced map. From \cite[Lemma 2.3]{Grabiner}   it follows that
\[k_n(T)=\dim (N(T)\cap R(T^n))/( N(T)\cap R(T^{n+1}))=\dim (R(T)+N(T^{n+1}))/(R(T)+N(T^n)).\] From this it is easily
seen that
$
k_n(T)=c_n^{\prime}(T)-c_{n+1}^{\prime}(T) $ if $c_{n+1}^{\prime}(T)<\infty$ and $ k_n(T)=c_n(T)-c_{n+1}(T)$ if $c_{n+1}(T)<\infty$.

The  {\it descent} $\delta(T)$ and the {\it ascent} $ a(T) $ of $T$ are defined by
 $ \delta(T)=\inf \{ n\in\NN_0:c_{n}(T)=0 \}=\inf  \{n\in\NN_0: R(T^n) = R(T^{n+1})\}$
and
 $ a(T)=\inf \{ n\in\NN_0:c^\prime_{n}(T)=0 \}= \inf \{n\in\NN_0 : N(T^{n})=N(T^{n+1})\}$. We set formally $\inf\emptyset =\infty$.

The {\it essential  descent} $\delta_e(T)$ and the {\it essential ascent} $ a_e(T) $ of $T$ are defined by
 $ \delta_e(T)=\inf \{ n\in\NN_0:c_{n}(T)<\infty \}$
and
 $ a_e(T)=\inf \{ n\in\NN_0:c^\prime_{n}(T)<\infty \}$.

The sets  of upper semi-Browder, lower semi-Browder and  Browder operators are
defined as $\B_+(X)=\{ T\in\Phi_+ (X): a (T)<\infty\}$, $\B_-(X)=\{ T\in\Phi_- (X): \delta(T)<\infty\}$  and  $\B(X)=\B_+(X)\cap \B_-(X)$, respectively.

Sets of {\it left and
right Drazin invertible} operators, respectively, are defined as
 $
LD(X) = \{T\in L(X) : a(T) < \infty {\rm \  and\ } R(T^{a(T)+1})\ {\rm  is\ closed }\}
$
 and
         $
RD(X) = \{T\in L(X) : \delta(T) < \infty {\rm \  and\ } R(T^{\delta(T)})\ {\rm  is\ closed }\}.
$
 If $   a(T) < \infty $ and $\delta(T)<\infty$, then $T$ is called  {\it Drazin invertible}. By $D(X)$ we denote the set of Drazin invertible operators.

An operator $T\in L(X)$ is a {\it left essentially  Drazin invertible} operator if
 $
 a_e(T) < \infty$ and $ R(T^{a_e(T)+1})$  is closed.
 If
        $\delta_e(T) < \infty$  and $ R(T^{\delta_e(T)})$  is closed, then $T$ is called {\it right essentially  Drazin invertible}.
In the sequel   $LD^e(X)$ (resp.
$RD^e(X)$)  will denote the set of left (resp. right) essentially  Drazin invertible
 operators.
For a bounded linear operator T and $n\in\NN_0$ define $T_n$ to be the
restriction of $T$ to $R(T^n)$ viewed as a map from $R(T^n)$ into $R(T^n)$ (in particular,
$T_0 = T$). If $T \in L(X) $ and if there exists an integer  $n$
 for which the range space $R(T^n)$ is closed and $T_{n}$  is   Fredholm (resp. upper semi-Fredholm, lower semi-Fredholm, Browder, upper semi-Browder, lower semi-Browder), then $T$ is called a {\it B-Fredholm}  (resp. {\it upper
semi-B-Fredholm, lower semi-B-Fredholm}, {\it B-Browder}, {\it upper
semi-B-Browder, lower semi-B-Browder} ) operator. If $T \in L(X)$ is upper or lower semi-B-Fredholm, then T is called
{\it semi-B-Fredholm}.
The index $\ind(T)$ of a semi-B-Fredholm
operator T is defined as the index of the semi-Fredholm operator $T_n$. By
\cite[Proposition 2.1]{Ber} the definition of the index is independent of the integer n.
An operator $T \in L(X)$
is {\it  B-Weyl} (resp. {\it  upper semi-B-Weyl, lower semi-B-Weyl}) if $T$ is   B-Fredholm
and ind(T) = 0 (resp. $T$ is upper semi-B-Fredholm and $\ind(T) \le 0$, $T$ is lower semi-B-
Fredholm and $\ind(T) \ge 0$).

We use the following notation:
\begin{center}
\begin{tabular}{|c|c|c|} \hline
${\bf R_1}: \J(X)$ & ${\bf R_2}: \S(X)$ & ${\bf R_3}: L(X)^{-1}$ \\
\hline
${\bf R_4}=\Phi_+(X)$ & ${\bf R_5}=\Phi_-(X)$ & ${\bf R_6}=\Phi(X)$ \\
\hline
${\bf R_{7}}=\mathcal{W}_+(X)$ & ${\bf R_{8}}=\mathcal{W}_-(X)$ & ${\bf R_{9}}=\mathcal{W}(X)$ \\
\hline
\end{tabular}
\end{center}
We recall that \cite[Theorem 3.6]{Ber0}
 \begin{eqnarray}
   &&{\bf BR_1}(X)={\bf B\B_+}(X)=LD(X),\quad {\bf BR_2}(X)={\bf B\B_-}(X)=RD(X),\nonumber\\
     &&{\bf BR_3}(X)={\bf B\B}(X)=D(X)\label{poziv} \\
   &&{\bf B}\Phi_+(X)=LD^e(X),\quad  {\bf B}\Phi_-(X)=RD^e(X).\nonumber
   \end{eqnarray}

 For $T\in L(X)$, the {\it upper semi-B-Fredholm spectrum}, the {\it  lower semi-B-Fredholm spectrum}, the {\it  B-Fredholm spectrum}, the {\it  upper semi-B-Weyl spectrum}, the {\it  lower semi-B-Weyl spectrum}, the {\it  B-Weyl spectrum}, the {\it  upper semi-B-Browder spectrum}, the {\it  lower semi-B-Browder spectrum} and  the {\it  B-Browder spectrum} are defined, respectively, by:
 \begin{eqnarray*}
      \sigma_{B\Phi_+}(T) &=& \{\lambda\in\CC:T-\lambda I\ \mbox{ is\ not\ upper\ semi-B-Fredholm}\}, \\
    \sigma_{B\Phi_-}(T) &=& \{\lambda\in\CC:T-\lambda I\ \mbox{ is\ not\ lower\ semi-B-Fredholm}\},\\
    \sigma_{B\Phi}(T) &=& \{\lambda\in\CC:T-\lambda I\ \mbox{ is\ not\ B-Fredholm}\},\\
     \sigma_{B\W_+}(T) &=& \{\lambda\in\CC:T-\lambda I\ \mbox{ is\ not\ upper\ semi-B-Weyl}\}, \\
     \sigma_{B\W_-}(T) &=& \{\lambda\in\CC:T-\lambda I\ \mbox{ is\ not\ lower\ semi-B-Weyl}\}, \\
     \sigma_{B\W}(T) &=& \{\lambda\in\CC:T-\lambda I\ \mbox{ is\ not\ B-Weyl}\},\\
     \sigma_{B\B_+}(T) &=& \{\lambda\in\CC:T-\lambda I\ \mbox{ is\ not\ upper\ semi-B-Browder}\}, \\
     \sigma_{B\B_-}(T) &=& \{\lambda\in\CC:T-\lambda I\ \mbox{ is\ not\ lower\ semi-B-Browder}\}, \\
     \sigma_{B\B}(T) &=& \{\lambda\in\CC:T-\lambda I\ \mbox{ is\ not\ B-Browder}\}.
 \end{eqnarray*}


For $T \in L(X)$ and every $d\in\NN_0$, the operator range topology on $R(T^d)$ is defined by the norm $\|\cdot\|_d$  such that for every $y\in R(T^d)$,
$$
\|y\|_d=\inf\{\|x\|:x\in X,\ y=T^dx\}.
$$

Operators which have eventual topological uniform descent were introduced by Grabiner in \cite{Grabiner}:
\begin{definition} \rm
 Let $T \in L(X)$. If there is $d\in\NN_0$ for which $k_n(T)=0$ for $n\ge d$, then $T$ is said to have uniform descent for $n\ge d$. If in addition, $R(T^n)$ is closed in the operator range topology of $R(T^d)$ for $n\ge d$, then we say that $T$ has {\it eventual topological uniform descent} and, more precisely, that $T$ has {\it  topological uniform descent ({\rm TUD for brevity})  for} $n\ge d$.
\end{definition}

For $T\in L(X)$,    the {\it topological uniform descent spectrum} is defined by:
\begin{eqnarray*}
\sigma_{TUD}(T)=\{\lambda \in \CC: T-\lambda \ \text{ does\ not\  have\ TUD}\}.
\end{eqnarray*}

If $M$ is a subspace of $X$ such that $T(M) \subset M$, $T \in
L(X)$, it is said that $M$ is {\em $T$-invariant}. We define $T_M:M
\to M$ as $T_Mx=Tx, \, x \in M$.  If $M$ and $N$ are two closed
$T$-invariant subspaces of $X$ such that $X=M \oplus N$, we say that
$T$ is {\em completely reduced} by the pair $(M,N)$ and it is
denoted by $(M,N) \in Red(T)$. In this case we write $T=T_M \oplus
T_N$ and say that $T$ is the {\em direct sum} of $T_M$ and $T_N$.

For $T \in
L(X)$ we say that it is {\em Kato} if $R(T)$ is closed and $N(T) \subset
R(T^n)$ for every $ n \in \mathbb{N}$. 
It is said that $T \in L(X)$ admits a Kato decomposition or $T$ is
of Kato type if there exist two closed $T$-invariant subspaces $M$
and $N$ such that $X=M \oplus N$, $T_M$ is Kato and $T_N$ is
nilpotent.
If we require that $T_N$ is quasinilpotent instead of nilpotent in
the definition of the Kato decomposition, then it leads us to the
generalized Kato decomposition, abbreviated as GKD \cite{pfred}. 
 An operator $T \in L(X)$ is said to  admit a {\em generalized Kato-Riesz
decomposition} if there exists a pair $(M,N)
\in Red(T)$ such that $T_M$ is Kato and $T_N$ is Riesz.

For $T\in L(X)$,    the {\it Kato type  spectrum}, the {\it  generalized Kato spectrum}
 and the {\it  generalized Kato-Riesz spectrum} are defined, respectively, by:
\begin{eqnarray*}\sigma_{Kt}(T)&=&\{\lambda \in
\CC: T-\lambda\ \text{is\ not\ of\ Kato\ type}\},\\
\sigma_{gK}(T)&=&\{\lambda \in \CC: T-\lambda \ \text{does\ not\ admit\ a\ generalized\
Kato\ decomposition}\}\\
\sigma_{gKR}(T)&=&\{\lambda\in\CC: T-\lambda\ \text{does\ not\ admit\  a\ generalized\ Kato-Riesz\
decomposition}\}
.
\end{eqnarray*}

If $K \subset \mathbb{C}$, then $\partial K$ is the
boundary of $K$, $\acc \, K$ is the set of accumulation points of
$K$, $\iso \, K=K \setminus \acc \, K$ and $\inter\, K$ is the set of interior points of $K$.
For $\lambda_0\in \CC$, the open disc, centered  at $\lambda_0$  with radius $\epsilon$ in $\CC$, is denoted  by $D(\lambda_0,\epsilon)$.

If $T\in L(X)$ the {\it reduced minimum modulus} of a non-zero operators $T$ is defined to be
$$
\gamma(T)=\inf\limits_{x\notin N(T)}\ds\frac {\|Tx\|}{dist(x, N(T))}.
$$
If $T=0$, then we take $\gamma(T)=\infty$.

\medskip

Recall that an operator $T\in L(X)$ is Drazin invertible if there is $S\in L(X)$ such that
\begin{equation*}
  TS=ST,\ STS=S, \ TST-T\ {\rm is\ nilpotent}.
\end{equation*}
The concept of the generalized Drazin invertible operators was
introduced by J. Koliha \cite{koliha}: an operator $T\in L(X)$  is generalized Drazin
invertible  in case there is $S\in L(X)$ such that
\begin{equation*}
  TS=ST,\ STS=S, \ TST-T\ {\rm is\ quasinilpotent}.
\end{equation*}
 Recall that $T$ is generalized Drazin invertible if and only if $0\notin\acc\, \sigma (T)$, and this is also equivalent to the fact that
 $T=T_1\oplus T_2$ where $T_1$ is invertible and $T_2$ is quasinilpotent. In \cite{SZC} this concept is further generalized by   replacing the third condition in the previous definitions  by the condition that $TST-T$ is Riesz,  and so it is introduced the concept of generalized Drazin-Riesz invertible operators. It is proved that $T$ is generalized Drazin-Riesz invertible if and only if $T$ admits a generalized Kato-Riesz
decomposition and $0$ is not an interior point of $\sigma (T)$, and this is also equivalent to the fact that $T=T_1\oplus T_2$, where $T_1$ is invertible and $T_2$ is Riesz.

 In this paper we further generalize  this concept by introducing  generalized Drazin-meromorphic
invertible operators.
  \begin{definition}
An operator $T\in L(X)$ is  {\em generalized Drazin-meromorphic
invertible}, if there exists $S \in L(X)$ such that
\[TS=ST, \; \; \; STS=S, \; \; \; TST-T \; \; \text{is meromorphic}.\]
\end{definition}
Obviously, every meromorphic operator is generalized Drazin-meromorphic
invertible.
\begin{definition}
 An operator $T \in L(X)$ is said to  admit a {\em generalized Kato-meromorphic
decomposition}, abbreviated to $GK(\M)D$, if there exists a pair $(M,N)
\in Red(T)$ such that $T_M$ is Kato and $T_N$ is meromorphic (i.e. $T_N\in (\M)$). In that case we shall say that $T$ admits a $GK(\M)D(M,N)$.
\end{definition}
\begin{definition} An operator $T\in L(X)$  satisfies $T \in {\bf GD(\M)R_i}$ if there exist $(M,N)\in Red(T)$ such that $T_M\in{\bf R_i}$ and $T_N\in (\M)$, $1\le i\le 9$.
\end{definition}

In the second section of this paper,  using   Grabiner's punctured neighborhood theorem \cite[Theorem 4.7]{Grabiner}, we characterize operators which belong to the set ${\bf GD(\M)R_i}$,  $1\le i\le 9$. In particular, we characterize  generalized Drazin-meromorphic invertible operators and among other results, we prove that $T$ is generalized Drazin-meromorphic
invertible if and only if $T=T_1\oplus T_2$ where $T_1$ is invertible and $T_2$ is meromorphic,  and this is also equivalent to the fact that
$T$ admits a generalized Kato-meromorphic
decomposition and $0$ is not an interior point of $\sigma (T)$. Also, if $T$ admits a generalized Kato-meromorphic
decomposition, then $0$ is not an interior point of $\sigma (T)$ if and only if $0$ is not an acumulation point of $\sigma_{\bf BB}(T)$.

An operator $T\in L(X)$ is said to have the single-valued extension property at $\lambda_0\in\CC$ (SVEP at $\lambda_0$ for breviety) if for every open disc $\D_{\lambda_0}$ centerd at $\lambda_0$ the only analitic function $f:\D_{\lambda_0}\to X$ satisfying $(T-\lambda)f(\lambda)=0 $ for all $\lambda\in \D_{\lambda_0}$ is the function $f\equiv 0$.
An operator $T\in L(X)$ is said to have the SVEP if $T$ has the SVEP at every point $\lambda\in\CC$.
There are implications (see \cite{aienarosas}, p. 182):
\begin{eqnarray}
  &&\sigma(T)\ {\rm does\ not\ cluster\ at\ }\lambda_0\Longrightarrow T\ {\rm and\ }T^\prime\ {\rm have\ the\ SVEP\ at\ }\lambda_0,\label{w1}\\
  &&\sigma_{ap}(T)\ {\rm does\ not\ cluster\ at\ }\lambda_0\Longrightarrow T\ {\rm has\ the\ SVEP\ at\ }\lambda_0,\label{w2}\\
  &&\sigma_{su}(T)\ {\rm does\ not\ cluster\ at\ }\lambda_0\Longrightarrow T^\prime\ {\rm has\ the\ SVEP\ at\ }\lambda_0.\label{w3}
\end{eqnarray}

P. Aiena and E. Rosas gave characterizations of the SVEP at $\lambda_0$ in the case that $\lambda_0-T$ is of Kato type. Precisely, they proved that if $\lambda_0-T$ is of Kato type, then the implications \eqref{w1}, \eqref{w2} and \eqref{w3} can be reversed \cite{aienarosas}. Q. Jiang and H. Zhong \cite{kinezi} gave further characterizations of the SVEP at $\lambda_0$ in the case that $\lambda_0-T$ admits a generalized Kato decomposition. They  proved that if $\lambda_0-T$ admits a GKD,  then the following statements are equivalent:

(i) $T$\ ($T^\prime$) has the SVEP at $\lambda_0$;

(ii) $\sigma_{ap}(T)$ ($\sigma_{su}(T)$) does not cluster at $\lambda_0$;

(iii) $\lambda_0$ is not an interior point of $\sigma_{ap}(T)$ ($\sigma_{su}(T)$),

\noindent that is, the implications \eqref{w1}, \eqref{w2} and \eqref{w3}   can be also reversed in the case that $\lambda_0-T$ admits a GKD.
In \cite{SZC} it was showed that  if $\lambda_0-T$ admits a generalized Kato-Riesz
decomposition,  then the following statements are equivalent:

(i) $T\ (T^\prime)\ {\rm has\ the\ SVEP\ at\ }\lambda_0$;

(ii) $\sigma_{\B_+}(T)\ (\sigma_{\B_-}(T))$ does not cluster at $\lambda_0$;

(iii) $\lambda_0$ is not an interior point of $\sigma_{ap}(T)\ (\sigma_{su}(T))$.

In the second section we give further characterizations of the SVEP at $\lambda_0$ always in the case that $\lambda_0-T$ admit a generalized Kato-meromorphic
decomposition. Precisely,  we prove that if $\lambda_0-T$ admits a generalized Kato-meromorphic
decomposition, then $T$ has the SVEP at $\lambda_0$ if and only if $\sigma_{\bf BB_+}(T)$ does not cluster at $\lambda_0$, and it is precisely when $\lambda_0$ is not an interior point of $\sigma_{ap} (T)$. A dual result shows  that, always if $\lambda_0-T$ admits a generalized Kato-meromorphic
decomposition, $T^\prime$ has the SVEP at $\lambda_0$ if and only if $\sigma_{\bf BB_-}(T)$ does not cluster at $\lambda_0$, and it is precisely when $\lambda_0$ is not an interior point of $\sigma_{su} (T)$.


Also,  we prove that if
$\lambda_0-T$ admits a $GK(\M)D$,  then  $\lambda_0$ is not an interior point of $\sigma_{{\bf R}}(T)$
if and only if
 $\sigma_{{\bf BR}}(T)$ does not cluster at $\lambda_0$ where ${\bf R}$ is one of $\Phi_+, \Phi_-, \Phi, \W_+, \W_-, \W$.

In the third section    we investigate  corresponding spectra.
 For $T\in L(X)$, the generalized Drazin-meromorphic spectrum and the generalized Kato-meromorphic spectrum
are
defined, respectively, by
\begin{eqnarray*}
  \sigma_{\bf gD\M}(T) &=& \{\lambda\in\CC: T-\lambda\ \mbox{is\ not\
generalized\ Drazin-meromorphic\ invertible}\} \\
  \sigma_{\bf gK \M}(T) &=&  \{\lambda \in \mathbb{C}:T-\lambda   \mbox{ does\ not\ admit\ generalized\ Kato-meromorphic\ decomposition}\}.
\end{eqnarray*}
It is proved that these spectra are compact and that  the generalized Kato-meromorphic spectrum differs from the Kato type spectrum                              on at most countably many points. We deduce several results  on cluster points of some semi-B-Fredholm spectra. Among other results, it is proved that $ \partial \, \sigma_{{\bf  R}_i}(T)\cap \acc\, \sigma_{{\bf  BR}_i}(T)\subset \partial\, \sigma_{\bf g K\M}(T)$, $1\le i\le 9$. Also we get some results regarding boundaries and connected  hulls of corresponding spectra (${\bf GD(\M)R_i}$ spectra), and
obtain that  the
generalized Drazin-meromorphic spectrum and  the generalized Kato-meromorphic spectrum of $T\in L(X)$
 are empty in the same time and that this happens
  if and only if $T$ is polynomially meromorphic.

  These results are  applied  to some concrete operators, amongst them the unilateral weighted right shift operator on $\ell_p(\NN)$, $1\le p<\infty$, the forward and backward unilateral shifts on $  c_0(\NN), c(\NN), \ell_{\infty}(\NN)$ or $\ell_p(\NN)$, $1\le p<\infty$,   arbitrary non-invertible isometry, and  $\rm Ces\acute{a}ro$ operator.

 \section{Generalized Drazin-meromorphic invertible and\\generalized  Drazin-meromorphic semi-Fredholm\\ operators}
We start with the following auxiliary assertions.
\begin{lemma}\label{Kato-SVEP}
Let $T\in L(X)$. If $T$ is Kato and $T$, $T^\prime$ have SVEP at 0, then $T$ is invertible.
\end{lemma}
\begin{proof} It follows from \cite[Theorem 2.49]{Ai}.
\end{proof}
\begin{lemma}\label{lema-meromorphic}
Let $T\in L(X)$ and let $(M,N)\in Red (T)$. Then
  $T\in (\M)$ if and only if $T_M\in (\M)$ and $T_N\in (\M)$.
\end{lemma}
\begin{proof}
Since $\sigma_D(T)=\sigma_D(T_M)\cup \sigma_D(T_N)$, it follows that $\sigma_D(T)\subset \{0\}$ if and only if $\sigma_D(T_M)\subset \{0\}$ and $\sigma_D(T_N)\subset \{0\}$. Consequently, $T$ is meromorphic if and only if $T_M$ and $T_N$ are meromorphic.
\end{proof}
\begin{lemma}\label{suma}
Let $T \in L(X)$ and $(M,N) \in Red(T)$. The following statements
hold:
\par \noindent {\rm (i)} $T \in {\bf BR}_i$,  $1 \leq i \leq 6$, if and only if $T_M
\in {\bf BR}_i$ and $T_N \in {\bf BR}_i$ and  then $\ind (T)=\ind(T_M)+\ind(T_N)$;
\par \noindent {\rm (ii)}
If $7 \leq i \leq 9$, $T_M \in {\bf BR}_i$ and $T_N \in {\bf BR}_i$,  then $T \in {\bf
BR}_i$.

\snoi{\rm (iii)} If  $7 \leq i \leq 9$, $T \in {\bf BR}_i$ and $T_N$ is B-Weyl, then
$T_M \in {\bf BR}_i$.
\end{lemma}
\begin{proof}
Since  $N(T^n)=N(T_M^n)\oplus N(T_N^n)$ and
$R(T^n)=R(T_M^n)\oplus R(T_N^n)$ for every $n\in\NN_0$, it follows that $c_n(T)=c_n(T_M)+c_n(T_N)$ and $c_n^\p(T)=c_n^\p(T_M)+c_n^\p(T_N)$. Thus  $c_n(T)=0$ if and only if $c_n(T_M)=c_n(T_N)=0$, and  $c_n^\p(T)=0$ if and only if $c_n^\p(T_M)=c_n^\p(T_N)=0$. Also $c_n(T)<\infty$ if and only if $c_n(T_M)<\infty$, $c_n(T_N)<\infty$, and $c_n^\p(T)<\infty$ if and only if $c_n^\p(T_M)<\infty$, $c_n^\p(T_N)<\infty$. Therefore, $\delta(T)<\infty$ if and only if $\delta(T_M)<\infty$ and $\delta(T_N)<\infty$,  in which  case $\delta(T)=\max\{\delta(T_M), \delta(T_N)\}$, and similarly for the  ascent, the essential descent and the essential ascent. As $R(T^n)$ is closed if and only
if $R(T_M^n)$ and $R(T_N^n)$ are closed \cite[Lemma 3.3]{kinezi}, for every $n\in\NN_0$,  we get that  $T$ is left  Drazin invertible (right  Drazin invertible, left (right) essentially  Drazin invertible)  if and only if $T_M$ and $T_N$ are left  Drazin invertible (right  Drazin invertible, left (right) essentially  Drazin invertible), which by \cite[Theorem 3.6]{Ber0} means that $T \in {\bf BR}_i$,  $1 \leq i \leq 6$, if and only if $T_M
\in {\bf BR}_i$ and $T_N \in {\bf BR}_i$. In that case, for $n$ large enough, we have that
\begin{eqnarray*}
  \ind (T) &=& c_n^\p(T)- c_n(T)=c_n^\p(T_M)+c_n^\p(T_N)-(c_n(T_M)+c_n(T_N))\\
   &=& (c_n^\p(T_M)-c_n(T_M)) + (c_n^\p(T_N)-c_n(T_N))= \ind (T_M)+\ind (T_N).
\end{eqnarray*}

(ii) follows from (i).

(iii):  Suppose that $T$ is upper semi-B-Weyl and that $T_N$ is B-Weyl. Then $\ind (T_N)=0$ and from {\rm (i)} it follows that $T_M$ is upper semi-B-Fredholm  and $\ind (T_M)=\ind(T_M)+\ind(T_N)=\ind (T)\le 0$. Hence $T_M$ is upper semi-B-Weyl.

Similarly for the rest of the cases.
\end{proof}
\begin{lemma}\label{zatvorenost} Let $X=X_1\oplus X_2\dots\oplus X_n$ where $X_1,\ X_2,\dots ,X_n$ are closed subspaces of $X$ and let $M_i$ be a closed subset of $X_i$, $i=1,\dots,n$. Then the set $M_1\oplus M_2\oplus\dots\oplus M_n$ is   closed.
\end{lemma}
\begin{proof}
Consider Banach space $X_1\times X_2\times\dots\times X_n$ provided with the canonical norm $\|(x_1,\dots, x_n)\|=\sum_{i=1}^n\|x_i\|$, $x_i\in X_i$, $i=1,\dots,n$. Then the map  $f:X_1\times\dots\times X_n\to X_1\oplus\dots\oplus X_n=X$    defined by $f((x_1,\dots,x_n)) =x_1+\dots +x_n$, $x_i\in X_i$, $i=1,\dots,n$, is a
  homeomorphism. Since $M_1\times M_2\times\dots\times M_n$ is closed in  $X_1\times X_2\times\dots\times X_n$, it follows that $f(M_1\times M_2\times\dots\times M_n)=M_1\oplus M_2\oplus\dots\oplus M_n$ is closed.
\end{proof}

\begin{lemma} \label{gap0}
Let $T \in L(X)$ and $1\le i\le 9$. Then the following implication holds:

$$T\ {\rm is\ Kato}\ \ \mbox{and}\  \ 0\notin\inter\, \sigma_{\bf BR_i}(T)\Longrightarrow T\in {\bf R_i}.$$
%
%
%
%
%
%
\end{lemma}
\begin{proof} 
We prove the assertion for the cases $i=1$ and $i=6$.

 Suppose that $T$ is Kato and that $0$ is not an interior point of $\sigma_{\bf BR_1}(T)$. Then $T$ has TUD for $n\ge 0$,  and from \cite[Theorem 4.7]{Grabiner} we have that  there is an $\epsilon>0$ such that for every $\lambda\in\CC$ the following implication holds:
      \begin{equation}\label{bh1}
    0<|\lambda|<\epsilon\Longrightarrow c_n^\prime(T-\lambda I)=c_0^\prime(T)=\alpha(T)\ {\rm for\ all\ } n\ge 0.
\end{equation}
 From $0\notin \inter\, \sigma_{\bf BR_1}(T)$ it follows that there exists $\lambda_0\in\CC$ such that $0<|\lambda_0|<\epsilon$ and $T-\lambda_0 I$ is a left Drazin invertible  operator. This implies that there exists $n\in\NN_0$ such that $c_n^\prime(T-\lambda_0 I)=0$. From \eqref{bh1} it follows that $\alpha(T)=0$ and since $R(T)$ is closed, we obtain that  $T$ is bounded bellow.

%
%
%

  Suppose  that $T$ is Kato and that $0$ is not an interior point of $\sigma_{B\W_+}(T)$. According to \cite[Theorem 4.7]{Grabiner}   there is an $\epsilon>0$ such that for every $\lambda\in\CC$ the following implication  holds:
      \begin{eqnarray}
& 0<|\lambda|<\epsilon \Longrightarrow\label{bh3}\\& c_n(T-\lambda I)=c_0(T)=\beta(T)\ {\rm and}\   c_n^\prime(T-\lambda I)=c_0^\prime(T)=\alpha(T)\ {\rm for\ all\ }n\ge 0.\nonumber
\end{eqnarray}
From $0\notin \inter\, \sigma_{B\W_+}(T)$, we have that there exists $\lambda_0\in\CC$ such that $0<|\lambda_0|<\epsilon$ and $T-\lambda_0 I$ is an upper semi-B-Weyl operator.
  Therefore,
    $$c_n^\prime(T-\lambda_0 I)=\dim (N(T- \lambda_0 I)\cap R((T-\lambda_0 I)^n))<\infty $$ for $n$ large enough and according to \eqref{bh3} we obtain that $\alpha(T)<\infty$.
     As $R(T)$ is closed, we get that  $T$ is upper semi-Fredholm.
    From
    \begin{eqnarray*}
  {\rm ind }(T)
  & =&c_0^\prime(T)- c_0(T)=c_n^\prime(T- \lambda_0 I)- c_n(T-\lambda_0  I)\\&=&{\rm ind}(T-\lambda_0  I)\le 0,
\end{eqnarray*}
 it follows that $T$ is an upper semi-Weyl operator.

  The remaining cases can be proved similarly.
\end{proof}

Let $U\subset X$ and $W\subset
X^\prime $. The annihilator of $U$ is the set $U^\bot=\{x'\in X':
x'(u)=0\text{\ for every\ }u\in U\}$, and the annihilator of $W$
is the set $^\bot W=\{x\in X: w(x)=0\text{\ for every\ }w\in
W\}$.
\begin{lemma}\label{help} Let $(M,N)\in Red (T)$. Then
$T$ admits a $GK(\M)D(M,N)$ if and only if $T^\prime$ admits a $GK(\M)D(N^\bot,M^\bot)$.
\end{lemma}
\begin{proof}
Suppose that $T$ admits a $GK(\M)D(M,N)$. Then  $T_M$ is Kato, $T_N\in (\M)$ and  $(N^\bot, M^\bot)\in Red (T^\prime)$.
 Let $P_N$ be the projection  of $X$ onto $N$ along $M$. Then $(M,N)\in Red (TP_N)$,  $TP_N=P_NT$, $TP_N=0\oplus T_N$, and so,  according to Lemma \ref{lema-meromorphic}, it follows that  $TP_N\in (\M)$.
\ Consequently, $T^\prime P_N^\prime=P_N^\prime T^\prime\in (\M)$, $(N^\bot, M^\bot)\in Red (T^\prime P_N^\prime )$ and since  $R(P_N^\prime)=N(P_N)^\bot=M^\bot$,  we conclude that $(T^\prime P_N^\prime)_{M^\bot}={T^\prime}_{M^\bot}\in (\M)$  according to Lemma \ref{lema-meromorphic}. From the proof of Theorem 1.43 in \cite{Ai} it follows  that ${T^\prime}_{N^\bot}$ is Kato. Therefore, $(N^\bot, M^\bot)$ is a $GK(\M)D$  for $T^\prime$.

Let $T^\prime$ admit a  $GK(\M)D(N^\bot,M^\bot)$. Then ${T^\prime}_{N^\bot}$ is Kato and ${T^\prime}_{M^\bot}\in (\M)$.  As $(N^\bot, M^\bot)\in Red (T^\prime P_N^\prime)$, then $T^\prime P_N^\prime=(T^\prime P_N^\prime)_{N^\bot}\oplus(T^\prime P_N^\prime)_{M^\bot}=0\oplus {T^\prime}_{M^\bot}$, and according to Lemma \ref{lema-meromorphic} we obtain  $T^\prime P_N^\prime\in (\M)$, which implies $TP_N\in (\M)$. Since $TP_N=0\oplus T_N$, from Lemma \ref{lema-meromorphic} we get $T_N\in (\M)$. Let $P_M=I-P_N$. Since ${T^\prime}_{N^\bot}$ is Kato, it follows that $R(({T^\prime}_{N^\bot})^n)$ is closed  (\cite[Theorem 12.2]{Mu}) and
\begin{equation}\label{gs}
  N((T^\prime_{N^\bot})^n)\subset R(T^\prime_{N^\bot}), \ {\rm for\ every\ } n\in \NN.
\end{equation}

 From $(N^\bot, M^\bot)\in Red (T^\prime P_M^\prime)$, we have $(T^\prime P_M^\prime)^n=(T^\prime P_M^\prime)_{N^\bot}^n\oplus(T^\prime P_N^\prime)_{M^\bot}^n={T^\prime}_{N^\bot}^n\oplus 0$ and
\begin{equation}\label{fr0}
  R((T^\prime P_M^\prime)^n)=R(({T^\prime}_{N^\bot})^n).
\end{equation}
 So $R((T^\prime P_M^\prime)^n)$ is closed which implies that $R((TP_M)^n)=R(T_M^n)$ is closed. As
\begin{eqnarray*}
  N((T^\prime_{N^\bot})^n) &=& N((T^\prime)^n)\cap N^\bot=R(T^n)^\bot\cap N^\bot=(R(T^n)+N)^\bot \\
   &=& (R(T_M^n)+R(T_N^n)+ N)^\bot= (R(T_M^n)+ N)^\bot,
\end{eqnarray*}
from (\ref{gs}) and (\ref{fr0}) we obtain
$$
(R(T_M^n)+ N)^\bot\subset R((TP_M)^\prime)=N(TP_M)^\bot, \ {\rm for\ every\ } n\in \NN,
$$
which implies
\begin{equation}\label{bez}
 N(TP_M)=\, ^\bot{(N(TP_M)^{\bot})}\subset\,  ^\bot{({(R(T_M^n)+ N)}^{\bot})} \ {\rm for\ every\ } n\in \NN.
\end{equation}
From Lemma \ref{zatvorenost} it follows  that $R(T_M^n)+ N $ is closed and therefore, according to \cite[Chapter III, Lemma 3.2]{Sch}, we get $^\bot{({(R(T_M^n)+ N)}^{\bot})}=R(T_M^n)+ N$, $n\in\NN$. Now from (\ref{bez}) we obtain
$$
N(T_M)+N\subset\, R(T_M^n)+ N \ {\rm for\ every\ } n\in \NN.
$$
It implies $N(T_M)\subset\, R(T_M^n)\ {\rm for\ every\ } n\in \NN$ and we can conclude that $T_M$ is Kato.
\end{proof}

\begin{theorem}\label{F2}
The following conditions are equivalent for $T\in L(X)$ and $1\le i\le 9$:

\par\noindent {\rm (i)} There exists $(M,N)\in Red(T)$ such that $ T_M\in {\bf R}_i$ and $T_N\in (\M)$, that is $T\in {\bf GD(\M)R_i}$;

\par\noindent {\rm (ii)} $T$ admits a $GK(\M)D$ and
$0\notin{\rm int}\, \sigma_{{\bf R}_i}(T)$;

\par\noindent {\rm (iii)} $T$ admits a $GK(\M)D$ and
$0\notin{\rm acc}\, \sigma_{{\bf BR}_i}(T)$;

\par\noindent {\rm (iv)} $T$ admits a $GK(\M)D$ and
$0\notin{\rm int}\, \sigma_{{\bf BR}_i}(T)$.
\end{theorem}
\begin{proof}
(i)$\Longrightarrow$(ii), (i)$\Longrightarrow$(iii): Suppose that there exists $(M,N)\in Red(T)$ such that $ T_M\in {\bf R}_i$ and $T_N\in (\M)$. 
For $1\le i\le 3$, $T_M$ is Kato, and so $T$  admits a $GK(\M)D$. For $4\le i\le 9$,
from
\cite[Theorem 16.20]{Mu} there exists $(M_1,M_2)\in Red(T_M)$ such that $\dim M_2<\infty$,  $T_{M_1}$ is Kato and $T_{M_2}$ is nilpotent. Then for $N_1=M_2\oplus N$ we have that  $N_1$ is a closed subspace  and  $T_{N_1}=T_{M_2}\oplus T_{N} \in(\M)$ by  Lemma \ref{lema-meromorphic}. So $T$  admits a $GK(\M)D$.

From $ T_M\in {\bf R}_i$   it follows that $0\notin  \sigma_{{\bf R}_i}(T_M)$ and there exists an $\epsilon>0$ such that $D(0,\epsilon)\cap \sigma_{{\bf R}_i}(T_M)=\emptyset$. As   $T_N\in (\M)$, we have that $\sigma_{{\bf R}_i}(T_N)$ is at most countable (with $0$ as its only possible limit point) and since $ \sigma_{{\bf R}_i}(T)\subset \sigma_{{\bf R}_i}(T_M)\cup \sigma_{{\bf R}_i}(T_N)$, we conclude that $0\notin{\rm int}\, \sigma_{{\bf R}_i}(T)$.

From $ T_M\in {\bf R}_i$ it follows that there exists $\epsilon>0$ such that for every $\lambda\in \CC$ satisfying $|\lambda|<\epsilon$ we have that $T_M-\lambda I_M\in {\bf R_i}\subset {\bf BR_i}$. Since $T_N\in (\M)$,  we have that  $T_N-\lambda I_N$ is Drazin  invertible (and hence belongs to ${\bf BR_i}$) for every $\lambda\in \CC$ such that $0<|\lambda|<\epsilon$.  According to Lemma \ref{suma}
 {\rm (i)},{\rm (ii)}, we obtain that $T-\lambda I\in {\bf BR_i}$ for every $\lambda\in\CC$ such that $0<|\lambda|<\epsilon$, and so $0\notin \acc\, \sigma_{{\bf BR}_i}(T)$.

(ii)$\Longrightarrow$(iv): It follows from the inclusion $\sigma_{{\bf BR_i}}(T)\subset\sigma_{{\bf R_i}}(T)$.

  (iii)$\Longrightarrow$(iv): It is obvious.

(iv)$\Longrightarrow$(i): Suppose that $T$ admits a $GK(\M)D$ and
$0\notin{\rm int}\, \sigma
_{{\bf BR}_i}(T)$. Then there exists a decomposition $(M,N)\in Red(T)$ such that $ T_M$ is Kato and $T_N\in (\M)$. From $0\notin{\rm int}\, \sigma
_{{\bf BR}_i}(T)$ it follows that for every $\epsilon>0$ there exists $\lambda\in\CC$ such that $0<|\lambda|<\epsilon$ and $T-\lambda I\in {\bf BR_i}$. Since $T_N\in (\M)$ we have that $T_N-\lambda I_N$ is Drazin invertible (and hence B-Weyl), and therefore from Lemma \ref{suma}  we get that $T_M-\lambda I_M\in {\bf BR_i}$. Thus $0\notin{\rm int}\, \sigma
_{{\bf BR}_i}(T_M)$. As $T_M$ is Kato, from Lemma \ref{gap0}  it follows that $T_M\in {\bf R_i}$.
\end{proof}
\begin{theorem}\label{F1}
The following conditions are equivalent for $T\in L(X)$ and $1\le i\le 9$:

\par \noindent {\rm
(i)} There exists $(M,N)\in Red(T)$ such that $ T_M\in {\bf BR}_i$ and $T_N\in (\M)$ ;

\par \noindent {\rm (ii)}
There exists a bounded projection $P$ on $X$ which commutes with $T$
such that $T+P\in {\bf BR}_i$  and $TP\in (\M)$;

\par \noindent {\rm (iii)}
There exists a bounded projection $P$ on $X$ which commutes with $T$
such that $T(I-P)+P\in {\bf BR}_i$  and $TP\in (\M)$.

\end{theorem}
\begin{proof} (i)$\Longrightarrow$(ii): Suppose that there exists $(M,N)\in Red(T)$ such that $ T_M\in {\bf BR}_i$ and $T_N\in (\M)$.  Then the projection $P\in L(X)$ such that $N(P)=M$ and $R(P)=N$ satisfies $TP=PT$. Since $T_N\in(\M)$, $TP=0\oplus T_N\in (\M)$ by Lemma \ref{lema-meromorphic}. Again, since $T_N\in(\M)$, it follows that $T_N+ I_N$ is Drazin invertible, i.e. B-Browder, and so $T+P=T_M\oplus (T_N+I_N)\in {\bf BR_i}$ according to Lemma \ref{suma}.

(ii)$\Longrightarrow$(iii): If (ii) holds, then for $M=(I-P)X$ and $N=PX$ from $TP\in (\M)$ it follows that $I+TP=I_M\oplus(I_N+T_N)$ is Drazin invertible, and so $I_N+T_N$ is Drazin invertible  by Lemma \ref{suma} (and hence B-Weyl). Since $T+P=T_M\oplus(I_N+T_N)\in {\bf BR}_i$, from Lemma \ref{suma} we obtain  that $T_M\in {\bf BR}_i$. Again by Lemma \ref{suma} it follows that $T(I-P)+P=T_M\oplus I_N\in  {\bf BR}_i$.


(iii)$\Longrightarrow$(i): If (iii) holds, then the closed subspace $M=N(P)$ and $N=R(P)$ define decomposition $(M,N)\in Red(T)$, and since $TP\in (\M)$ and $TP=0\oplus T_N$, from Lemma \ref{lema-meromorphic} it follows that $T_N\in(\M)$.
 As $T(I-P)+P=T_M\oplus I_N\in {\bf BR_i}$, from Lemma \ref{suma} it follows that   $T_M\in {\bf BR}_i$.
\end{proof}

\begin{theorem}\label{svod}
If $X=H$ is a Hilbert space, or $i=3$ or $6$ or  $9$, then  the  conditions in Theorems \ref{F1} and \ref{F2} are equivalent.
\end{theorem}
\begin{proof} It is enough to prove that the condition (i) in Theorem \ref{F1} implies the condition (i) in Theorem \ref{F2}.

 Suppose that there exists $(M,N)\in Red(T)$ such that $ T_M\in {\bf BR}_i$ and $T_N\in (\M)$ and suppose that either $X=H$, or, $i=3$ or $6$ or $9$ (and $X$ is a Banach space). From \cite[Theorem 2.7]{Ber}, \cite[Lemma 4.1]{Ber2} and  \cite[Theorem 3.12]{Ber0} it follows that there exists $(M_1,M_2)\in Red(T_M)$ such that $T_{M_1}\in {\bf R_i}$ and $T_{M_2}$ is nilpotent. Define $N_1=M_2\oplus N$. Then $N_1$ is a closed subspace by Lemma \ref{zatvorenost} and  $T_{N_1}=T_{N}\oplus T_{M_2}\in(\M)$ by  Lemma \ref{lema-meromorphic}.
\end{proof}

\begin{definition} An operator $T\in L(X)$ is {\it meromorphic quasi-polar} if there exists a bounded projection $Q$ satisfying
\begin{equation}\label{mer-qp}
  TQ=QT,\ T(I-Q)\in (\M),\ Q\in(L(X)T)\cap (TL(X)).
\end{equation}

\end{definition}
\begin{theorem}\label{glavna}
The following conditions are mutually equivalent for operators $T\in L(X)$:

\par\noindent {\rm (i)}
 There exists $(M,N)\in Red(T)$ such that $T_M$ is invertible and $T_N\in (\M)$;

\par\noindent {\rm (ii)} $T$ admits a $GK(\M)D$ and $0\notin {\rm int}\, \sigma (T)$;

\par\noindent {\rm (iii)}
 $T$ admits a $GK(\M)D$ and,  $T$ and  $T^\prime$ have SVEP at 0;

\par\noindent {\rm (iv)} $T$ is generalized Drazin-meromorphic
invertible;


\par\noindent {\rm (v)} $T$ is meromorphic quasi-polar;

\par\noindent {\rm (vi)} There exists a projection $P\in L(X)$ such that $P$ commutes with $T$, $TP\in (\M)$ and $T+P$ is B-Browder;

\par\noindent {\rm (vii)} There exists a projection $P\in L(X)$ which commutes with $T$  and such that $TP\in (\M)$ and $T(I-P)+P$ is B-Browder;

\par\noindent {\rm (viii)} There exists $(M,N)\in Red(T)$ such that $T_M$ is  B-Browder and $T_N\in (\M)$;

\par\noindent {\rm (ix)} $T$ admits a $GK(\M)D$ and $0\notin\acc\, \sigma_{\bf BB}(T)$;

\par\noindent {\rm (x)} $T$ admits a $GK(\M)D$ and $0\notin{\rm int}\, \sigma_{\bf BB}(T)$.
\end{theorem}
\begin{proof} According to   Theorem \ref{svod}, Theorem \ref{F2}, Theorem \ref{F1} and \eqref{poziv} it is sufficient to prove that
 (ii)$\Longrightarrow$(iii)$\Longrightarrow$ (iv)$\Longrightarrow$(v)$\Longrightarrow$(vi).
%

If (ii) holds, then $0\notin {\rm int}\, \sigma (T)$ implies that either $0\notin\sigma (T)$ or $0\in\partial\sigma (T)$. In both cases  $T$ and $T^\prime$ have SVEP at 0; hence (ii) implies (iii).

If (iii) holds, then $(M,N)\in Red(T))$, $T_N\in (\M)$,  $T_M$ is Kato and hence, $T^\prime_{N^\bot}$ is Kato by Lemma \ref{help}.  Since $T$ and $T^\prime$ have SVEP at 0, it follows that $T_M$ and $T^\prime_{N^\bot}$ also  have SVEP at 0, which implies that $T_M$ and $T^\prime_{N^\bot}$ are injective.  As in  the proof of \cite[Lemma 3.13]{Ai} it can be proved  that $T_M$ is surjective, and so $T_M$ is invertible. The operator $S=T_M^{-1}\oplus 0\in L(M\oplus N)$ satisfies
$$
ST=TS,\ STS=S,\ {\rm and}\ TST-T=0\oplus(-T_N)\in (\M).
$$
Thus (iii) implies (iv).

To prove (iv) implies (v), (assume (iv) and) define the projector $Q\in L(X)$ by setting $ST(=TS)=Q$. Then
\begin{equation}\label{proj}
 QT=TQ,\ Q\in TL(X)\cap L(X)T\ {\rm and}\ T(I-Q)\in (\M),
\end{equation}
i.e., $T$ is meromorphic quasi-polar.

Assume now that (v) holds. Then there exists a projector $Q\in L(X)$ such that (\ref{proj}) holds. Set $P=I-Q$. Then $TP\in (\M)$ and  for $N=P(X)$ and $M=(I-P)(X)$ we have
$$
PT=TP,\ T_N\in (\M)\ {\rm and}\ I-P=UT=TV
$$
for some $U,V\in L(X)$. Let $U,V\in B(M\oplus N)$have the ($2\times 2$ matrix) representations $U=[U_{ij}]_{i,j=1}^2$ and $V=[V_{ij}]_{i,j=1}^2$. Then
$$
\bmatrix U_{11}&U_{12}\\U_{21}&U_{22}\endbmatrix \bmatrix T_M&0\\0&T_N\endbmatrix=\bmatrix T_M&0\\0&T_N\endbmatrix\bmatrix V_{11}&V_{12}\\V_{21}&V_{22}\endbmatrix =\bmatrix I_M&0\\0&0\endbmatrix:(M\oplus N)\to (M\oplus N)
$$
and it follows from a straighforward calculation that ($T_M$ is invertible, $U_{21}=0=V_{12}$, $U_{12}T_N=U_{22}T_N=0=T_NV_{21}=T_NV_{22}$, and hence that) $UTV+P=T_M^{-1}\oplus I_N$ is invertible with
$(UTV+P)^{-1}=T_M\oplus I_N=T(I-P)+P$. Since $TP\in (\M)$, $I+TP$ is Drazin invertible and hence,
$$
T+P=(I+TP)(UTV+P)^{-1}=(UTV+P)^{-1}(I+TP)
$$
is Drazin invertible, i.e. B-Browder.
\end{proof}

The following theorems can be obtained similarly.

\begin{theorem}\label{glavna1}
The following conditions are mutually equivalent for operators $T\in L(X)$:

\par\noindent {\rm (i)} There exists $(M,N)\in Red(T)$ such that $T_M$ is bounded below and $T_N\in (\M)$;

\par\noindent {\rm (ii)} $T$ admits a $GK(\M)D$ and $0\notin {\rm int}\, \sigma_{ap} (T)$;

\par\noindent {\rm (iii)} $T$ admits a $GK(\M)D$ and  $T$ has SVEP at 0;

\par\noindent {\rm (iv)} $T$ admits a $GK(\M)D$ and $0\notin\acc\, \sigma_{\bf BB_+}(T)$;

\par\noindent {\rm (v)} $T$ admits a $GK(\M)D$ and $0\notin{\rm int}\, \sigma_{\bf BB_+}(T)$.
\end{theorem}

\begin{theorem}\label{glavna2}
The following conditions are mutually equivalent for operators $T\in L(X)$:

\par\noindent {\rm (i)} There exists $(M,N)\in Red(T)$ such that $T_M$ is surjective  and $T_N\in (\M)$;

\par\noindent {\rm (ii)} $T$ admits a $GK(\M)D$ and $0\notin {\rm int}\, \sigma_{su} (T)$;

\par\noindent {\rm (iii)} $T$ admits a $GK(\M)D$ and    $T^\prime$ has SVEP at 0;

\par\noindent {\rm (iv)} $T$ admits a $GK(\M)D$ and $0\notin\acc\, \sigma_{\bf BB_-}(T)$;

\par\noindent {\rm (v)} $T$ admits a $GK(\M)D$ and $0\notin{\rm int}\, \sigma_{\bf BB_-}(T)$.
\end{theorem}

We remark that if $T\in L(X)$ is Riesz with infinite spectrum, then  $T$ is generalized Drazin-meromorphic  invertible, $\sigma(T)=\sigma_{ap}(T)=\sigma_{su}(T)$,  $0\in{\rm acc}\, \sigma_{ap}(T)={\rm acc}\, \sigma_{su}(T)$ and $0\notin{\rm int}\, \sigma_{ap }(T)={\rm int}\, \sigma_{su }(T)$. Therefore, the condition that $0\notin{\rm int}\, \sigma_{ap }(T)$    ($0\notin{\rm int}\, \sigma_{su}(T)$) in the statement (ii) in Theorem \ref{glavna1} (Theorem \ref{glavna2}) can not be replaced with the stronger condition that $0\notin{\rm acc}\, \sigma_{ap}(T)$ ($0\notin{\rm acc}\, \sigma_{su}(T)$).

\bigskip

P. Aiena and E. Rosas \cite[Theorems 2.2 and 2.5]{aienarosas}   characterized  the SVEP at a point $\lambda_0$ in the case that $ \lambda_0-T$ is of Kato type. Q.  Jiang  and H. Zhong \cite[Theorems 3.5 and 3.9]{Kinezi} gave further characterizations of the SVEP at $\lambda_0$ in the case that $ \lambda_0-T$ admits a generalized Kato decomposition. We gave characterizations for the case that $ \lambda_0-T$ admits a generalized Kato-meromorphic decomposition.

\begin{corollary}  Let $T\in L(X)$ and let $\lambda_0-T$ admit a $GK(\M)D$. Then the  following statements are equivalent:

\par\noindent {\rm (i)} $T$ has the SVEP at $\lambda_0$;

\par\noindent {\rm (ii)} $\lambda_0$ is not an interior point of $\sigma_{ap} (T)$;

\par\noindent {\rm (iii)} $\sigma_{\bf BB_+}(T)$ does not cluster at $\lambda_0$.
\end{corollary}
\begin{proof} It follows from the equivalences (ii)$\Longleftrightarrow$(iii)$\Longleftrightarrow$(iv) in Theorem \ref{glavna1}.
\end{proof}

\begin{corollary}
 Let $T\in L(X)$ and let $\lambda_0-T$ admit a $GK(\M)D$. Then the  following statements are equivalent:

\par\noindent {\rm (i)} $T^\prime$ has the SVEP at $\lambda_0$;

\par\noindent {\rm (ii)} $\lambda_0$ is not an interior point of $\sigma_{su} (T)$;

\par\noindent {\rm (iii)} $\sigma_{\bf BB_-}(T)$ does not cluster at $\lambda_0$.
\end{corollary}
\begin{proof}
It follows from  Theorem \ref{glavna2}.
\end{proof}
\begin{theorem} The  following statements are equivalent for operator $T\in L(X)$:

\par\noindent {\rm (i)} $T=T_M\oplus T_N$, where $T_M$ is invertible, $T_N\in (\M)$ and $\sigma(T_N)$ is infinite;

\par\noindent {\rm (ii)} $T$ admits a $GK(\M)D$ such that there exists an infinite sequence of poles of resolvent of $T$ converging to $0$.

\end{theorem}
\begin{proof}
That  (i) implies (ii) is a straightforward consequence of the fact that $\sigma(T)=\sigma(T_M)\cup \sigma(T_N)$, where $\sigma(T_N)$ is a coutably infinite set with $\sigma_D(T_N)=\{0\}$.

(ii)$\Longrightarrow$(i): If (ii) holds, then there exists a decomposition $(M,N)\in Red(T)$ such that $T_M$ is  Kato and $T_N\in (\M)$. $T_N$ being meromorphc, $\sigma(T_N)$ is either finite or countably infinite (with $0$ as its only limit point). We prove that $\sigma(T_N)$ is finite leads to a contradiction. The hypotheses imply the existence of a sequence $(\mu_n)$ converging to $0$  such that $\mu_n\in \sigma(T)$ and $T-\mu_n$ is Drazin invertible for all $n\in\NN$. It means $0\notin {\rm int}\, \sigma_{\bf BB}(T)$ and an argument used  to prove that (x) implies (i) in Theorem \ref{glavna} shows that $T_M$ is invertible. The spectrum $\sigma(T_N)$ is finite, $\sigma(T)=\sigma(T_M)\cup \sigma(T_N)$ implies that a infinite number of the points $\mu_n\in \sigma(T_M)$, and hence $0\in \sigma(T_M)$-a contradiction. Hence $\sigma(T_N)$ is infinite.
\end{proof}

\begin{theorem} If  $T\in {\bf gD\M R_i}$ and $f$ is holomorphic in a neighbourhood of $\sigma (T)$ such that $f^{-1}(0)\cap \sigma_{\bf R_i}(T)=\{0\}$, then $f(T)\in {\bf gD\M R_i}$ for all $1\le i\le 9$.
\end{theorem}
\begin{proof} It is known that $f(\sigma_{\bf R_i}(T))=\sigma_{\bf R_i}(f(T))$ for all $f$ holomorphic on a neighbourhood of $\sigma (T)$ and $1\le i \le 6$. The corresponding inclusion for $7\le i\le 9$ is $\sigma_{\bf R_i}(f(T))\subset f(\sigma_{\bf R_i}(T))$. If $T\in {\bf gD\M R_i}$, then there exists a decomposition $(M,N)\in Red(T)$ such that $T_M\in {\bf R_i}$  and $T_N\in (\M)$. Furthermore $f(T)=f(T_M) \oplus f(T_N)$. Since $f(0)=0$, and since $f$ maps the poles of the resolvent of $T_N$ onto the poles of the resolvent of $f(T_N)$ (see the proof of the first part of \cite[Theorem 4.1]{BDHZ}), $f(T_N)\in (\M)$. Observe that $0\notin \sigma_{\bf R_i}(T_M)$ and since $f^{-1}(0)\cap \sigma_{\bf R_i}(T)=\{0\}$ we conclude that $0\notin f(\sigma_{\bf R_i}(T_M))$. This, since $f(\sigma_{\bf R_i}(T_M))\supset \sigma_{\bf R_i}(f(T_M))$ for all $1\le  i \le 9$, implies $0\notin \sigma_{\bf R_i}(f(T_M))$. This completes the proof.
\end{proof}

\section{Spectra}

For $T\in L(X)$, set
$$\sigma_{\bf gK \M}(T)=\{\lambda \in \mathbb{C}:T-\lambda   \mbox{ does\ not\ admit\ generalized\ Kato-meromorphic\ decomposition}\}$$
and
$$\sigma_{{\bf g D\M R}_i}(T)=\{\lambda \in \mathbb{C}:T-\lambda \not \in
{\bf g DR R}_i (X)\}, \; \; 1\leq i \leq 9.$$
In the following, we shorten $\sigma_{\bf gD\M L(X)^{-1}}(T)$ to
$$\sigma_{\bf g D\M}(T)=\{\lambda \in \mathbb{C}:T-\lambda \ \mbox{ is\ not\ generalized\ Drazin-meromorphic\ invertible}\}.$$
It is clear from Theorem \ref{F2} that
\begin{eqnarray}
 \sigma_{{\bf g D\M R}_i}(T)&=&\sigma_{\bf gK\M}(T) \cup {\rm int}\, \sigma_{{\bf
R}_i}(T)\label{glava}\\&=&\sigma_{\bf gK\M}(T) \cup  \acc \,
\sigma_{{\bf BR}_i}(T), \ 1\leq i \leq 9.\label{glava-}
\end{eqnarray}

\begin{theorem}\label{closed}
Let $T\in L(X)$ and let $T$ admits a $GK\M D(M,N)$. Then there exists $\epsilon>0$ such that $T-\lambda$ is of Kato type for each $\lambda$ such that  $0<|\lambda|<\epsilon$.
\end{theorem}
\begin{proof}
If $M=\{0\}$, then $T$ is meromorphic  and hence $T-\lambda$ is Drazin invertible  for all $\lambda\ne 0$. Therefore,  $T-\lambda$ is of  Kato type for all $\lambda\ne 0$.

Suppose that $M\ne \{0\}$. From \cite[Theorem 1.31]{Ai} it follows that for $|\lambda|<\gamma(T_M)$, $T_M-\lambda$ is Kato. As $T_N$ is meromorphic,  $T_N-\lambda$ is Drazin invertible, and hence it is of  Kato type for all $\lambda\ne 0$. Let $\epsilon=\gamma(T_M)$. According to \cite[p. 143]{MbekhtaMuller} it follows that $T-\lambda$ is of Kato type for each $\lambda$ such that  $0<|\lambda|<\epsilon$.
\end{proof}

\begin{corollary} \label{cor1}
Let $T\in L(X)$. Then

\snoi {\rm (i)} $\sigma_{\bf g K\M}(T)$  is compact;

\snoi {\rm (ii)} The set   $\sigma_{Kt}(T)\setminus\sigma_{\bf gK\M}(T)$ consists of  at most countably many points.

\end{corollary}
\begin{proof} (i): From Theorem \ref{closed} it follows that $\sigma_{\bf g K\M}(T)$  is closed and since $\sigma_{\bf g K\M}(T)\subset \sigma(T)$, $\sigma_{\bf g K\M}(T)$ is bounded. Thus $\sigma_{\bf g K\M}(T)$  is compact.

(ii): Suppose that $\lambda_0\in \sigma_{Kt}(T)\setminus\sigma_{\bf gK\M}(T)$. Then $T-\lambda_0$ admits a $GK\M D$ and according to Theorem \ref{closed} there exists $\epsilon>0$ such that $T-\lambda$ is of Kato type for each $\lambda$ such that  $0<|\lambda-\lambda_0|<\epsilon$. This implies that $\lambda_0\in\iso\, \sigma_{Kt}(T)$. Therefore,  $\sigma_{Kt}(T)\setminus\sigma_{\bf gK\M}(T)\subset \iso\, \sigma_{Kt}(T)$,   which implies that $\sigma_{Kt}\setminus\sigma_{\bf gK\M}(T)$ is  at most countable.
\end{proof}
\begin{corollary} \label{cor1*}
Let $T\in L(X)$ and $1\le i\le 9$. Then

\snoi {\rm (i)} $\sigma_{{\bf g D\M R}_i}(T)\subset \sigma_{{\bf  R}_i}(T)$;

\snoi {\rm (ii)} $\sigma_{{\bf g D\M R}_i}(T)$ is compact;

\snoi {\rm (iii)} $\inter \, \sigma_{{\bf g D\M R}_i}(T)=\inter\,  \sigma_{{\bf  R}_i}(T)$;

\snoi {\rm (iv)} $\partial \sigma_{{\bf g D\M R}_i}(T)\subset\partial \sigma_{{\bf  R}_i}(T)$;

\snoi {\rm (v)}  $\sigma_{{\bf  BR}_i}(T)\setminus\sigma_{{\bf g D\M R}_i}(T)=(\iso\, \sigma_{{\bf B R}_i}(T))\setminus\sigma_{\bf gK\M}(T)$;

\snoi {\rm (vi)} The set $\sigma_{{\bf  BR}_i}(T)\setminus\sigma_{{\bf g D\M R}_i}(T)$ consist of  at most countably many points.

\end{corollary}
\begin{proof} (i) Obvious.

(ii): From \eqref{glava-} and Corollary \ref{cor1} (i) it follows that  $\sigma_{{\bf g D\M R}_i}(T)$ is closed as the union of two closed sets, while from (i) it follows that   $\sigma_{{\bf g D\M R}_i}(T)$ is bounded, and so it is compact.

(iii):
From
\eqref{glava} we have  that $\inter\,  \sigma_{{\bf  R}_i}(T)\subset \sigma_{{\bf g D\M R}_i}(T)$, and hence $\inter\,  \sigma_{{\bf  R}_i}(T)\subset \inter\, \sigma_{{\bf g D\M R}_i}(T)$, while from the inclusion (i) it follows that $\inter \, \sigma_{{\bf g D\M R}_i}(T)\subset\inter\,  \sigma_{{\bf  R}_i}(T)$. Thus $\inter \, \sigma_{{\bf g D\M R}_i}(T)=\inter\,  \sigma_{{\bf  R}_i}(T)$.

(iv): Let $\lambda\in \partial \sigma_{{\bf g D\M R}_i}(T)$. Since $\partial \sigma_{{\bf g D\M R}_i}(T)\subset \sigma_{{\bf g D\M R}_i}(T)\subset \sigma_{{\bf  R}_i}(T) $,   from $\lambda\notin \inter \, \sigma_{{\bf g D\M R}_i}(T)=\inter\,  \sigma_{{\bf  R}_i}(T)$ we conclude $\lambda\in \partial \sigma_{{\bf  R}_i}(T)$. So, $\partial \sigma_{{\bf g D\M R}_i}(T)\subset\partial \sigma_{{\bf  R}_i}(T)$.

(v): It follows from  \eqref{glava-}.

(vi) It follows from (v).
\end{proof}

\medskip

From \eqref{glava} it follows that if $\sigma_{{\bf  R}_i}(T)$ is countable or contained in a line, then $\sigma_{{\bf g D\M R}_i}(T)=\sigma_{\bf g K\M}(T)$, $1\le i\le 9$.

\medskip

\begin{corollary}\label{presek} Let $T\in L(X)$ and $1\le i\le 9$. Then
\begin{equation}\label{presek1}
  \partial \, \sigma_{{\bf  R}_i}(T)\cap \acc\, \sigma_{{\bf  BR}_i}(T)\subset \partial\, \sigma_{\bf g K\M}(T).
\end{equation}
\end{corollary}
\begin{proof}
 Let $T-\lambda I$ admit a $GK(\M)D$ and let $\lambda\in \partial \, \sigma_{{\bf  R}_i}(T)$. Then $\lambda\notin \inter\, \sigma_{{\bf  R}_i}(T)$ and according to  the equivalence (ii)$ \Longleftrightarrow$(iii) in Theorem \ref{F2} it follows  that $\lambda\notin \acc\, \sigma_{{\bf  BR}_i}(T)$. Therefore,
\begin{equation}\label{presek2}
  \partial \, \sigma_{{\bf  R}_i}(T)\cap \acc\, \sigma_{{\bf  BR}_i}(T)\subset \sigma_{\bf g K\M}(T).
\end{equation}
Suppose that $\lambda\in \partial \, \sigma_{{\bf  R}_i}(T)\cap \acc\, \sigma_{{\bf  BR}_i}(T)$. Then  there exists a sequence $ (\lambda_n)$ which converges to $\lambda $ and  such that $T- \lambda_n\in {\bf R_i}$   for every $n\in\NN$. According to \cite[Theorem 16.21]{Mu} it follows that $T- \lambda_n$ admits a $GK(\M)D$, and so $\lambda_n\notin \sigma_{\bf g K\M}(T)$ for every $n\in\NN$. Since $\lambda\in \sigma_{\bf g K\M}(T)$ by \eqref{presek2}, we conclude that $\lambda\in \partial\, \sigma_{\bf g K\M}(T)$. This proves  the  inclusion \eqref{presek1}.
\end{proof}
Similarly to the inclusion \eqref{presek2}, the following inclusion can be proved
\begin{equation}\label{presek20}
  \partial \, \sigma_{{\bf B R}_i}(T)\cap \acc\, \sigma_{{\bf  BR}_i}(T)\subset \sigma_{\bf g K\M}(T),\ 1\le i\le 9.
\end{equation}

\begin{corollary}\label{SVEP}  Let $T\in L(X)$.

\snoi {\rm (i)} If $T$ has the SVEP, then all accumulation points of $\sigma_{\bf B\B_+}(T)$ belong to $\sigma_{\bf gK\M}(T)$.

\snoi {\rm (ii)} If $T^\prime$ has the SVEP, then all accumulation points of $\sigma_{\bf B\B_-}(T)$ belong to $\sigma_{\bf gK\M}(T)$.

\snoi {\rm (iii)} If  $T$ and $T^\prime$ have  the SVEP, then all accumulation points of $\sigma_{\bf B\B}(T)$ belong to $\sigma_{\bf gK\M}(T)$.
\end{corollary}
\begin{proof} (i): It follows from the equivalence (iii)$\Longleftrightarrow$(iv) of Theorem
\ref{glavna1}.

(ii): It follows from the equivalence (iii)$\Longleftrightarrow$(iv) of Theorem
\ref{glavna2}.

(iii): It follows from the equivalence (iii)$\Longleftrightarrow$(ix) of Theorem
\ref{glavna}.
\end{proof}

The next corollary extends \cite[Corollary 3.118]{Ai}.
\begin{corollary} \label{example1}
Let $T$ be unilateral weighted right shift operator on $\ell_p(\NN)$, $1\le p<\infty$, with weight $(\omega_n)$, and let $c(T)=\lim\limits_{n\to\infty}\inf (\omega_1\cdot\dots\cdot \omega_n)^{1/n}=0$. Then $\sigma_{\bf gK\M}(T)=\sigma_{\bf gD\M {\bf R}_i}(T)= \sigma(T)=\overline{D(0,r(T))}$, $1\le i\le 9$.
\end{corollary}
\begin{proof} According to \cite[Corollary 3.118]{Ai} it follows that
 $\sigma(T)=\overline{D(0,r(T))}$  and $T$ and $T^\prime$
have the SVEP. From the equivalence (ii)$\Longleftrightarrow$(iii) in Theorem \ref{glavna} it follows that $D(0,r(T))=\inter\, \sigma(T)\subset \sigma_{\bf gK\M}(T)$. Since $\sigma_{\bf gK\M}(T)$ is closed, we obtain that $\overline{D(0,r(T))}\subset \sigma_{\bf gK\M}(T)\subset\sigma_{\bf gD\M {\bf R}_i}(T)\subset\sigma(T)=\overline{D(0,r(T))}$, and so $\sigma_{\bf gK\M}(T)=\sigma_{\bf gD\M {\bf R}_i}(T)= \sigma(T)=\overline{D(0,r(T))}$.
\end{proof}

The {\it connected hull}  of a compact subset $K$ of the complex
plane $\CC$, denoted by $\eta K$, is the complement of the unbounded
component of $\CC\setminus K$ \cite[Definition
7.10.1]{H}.

We recall  that, for compact subsets $H,K\subset\CC$, the following implication holds (\cite[Theorem
7.10.3]{H}):
\begin{equation}\label{spec.1}
\partial H\subset K\subset
H\Longrightarrow\partial H\subset\partial K\subset K\subset
H\subset\eta K=\eta H\ .
\end{equation}
Evidently, if $K\subseteq\CC$ is at most countable, then $\eta K=K$.
Therefore, for compact subsets $H,K\subseteq\CC$, if
$\eta K=\eta H$, then $H\ {\rm is\
at\ most\ countable}$ if and only if $ K\ {\rm is\ at\ most\ countable}$,
and in that case $H=K$.

\medskip


\begin{theorem}\label{potrebna} Let $T\in L(X)$. Then

\snoi
{\rm (i)}
{\footnotesize
$$\begin{array}{ccccccccccc}
&&&&  \partial\,  \sigma_{\bf gD\M \J}(T) & \subset &    \partial\, \sigma_{\bf gD\M \W_+}(T) &  \subset  & \partial\, \sigma_{\bf gD\M \Phi_+}(T)& & \\
&&& \rotatebox{20}{$\subset$} & & \rotatebox{20}{$\subset$}&
&\!\rotatebox{20}{$\subset$}& & \rotatebox{-20}{$\subset$}&\\
  & &  \partial\, \sigma_{\bf gD\M}(T) & \subset &
\partial\, \sigma_{\bf gD\M\W}(T) & \subset &  \partial\, \sigma_{\bf gD\M \Phi}(T)&
  \subset  & &&  \partial\,  \sigma_{\bf gK\M}(T),\\
&&& \rotatebox{-20}{$\subset$} & & \rotatebox{-20}{$\subset$}&
&\rotatebox{-20}{$\subset$}& & \rotatebox{20}{$\subset$}&\\
& &&&   \partial\, \sigma_{\bf gD\M \Q}(T)& \subset& \partial\,  \sigma_{\bf gD\M \W_-}(T) & \subset & \partial\, \sigma_{\bf gD\M \Phi_-}(T) & &\\
\end{array}$$
}

\snoi {\rm (ii)}  $\eta\sigma_{\bf gD\M}(T)=\eta\sigma_{\bf gD\M R_i}(T)$, $1\le i\le 9$.

\snoi {\rm (iii)}
The set $\sigma_{\bf gD\M}(T)$ consists of $\sigma_{*}(T)$ and possibly some holes in $\sigma_{*}(T)$ where
 $\sigma_{*}\in\{\sigma_{\bf gK\M},\break\sigma_{\bf gD\M\W},\sigma_{\bf gD\M\Phi},\sigma_{\bf gD\M\W_+}, \sigma_{\bf gD\M\Phi_+}, \sigma_{\bf gD\M\J}, \sigma_{\bf gD\M\W_-},\sigma_{\bf gD\M\Phi_-}, \sigma_{\bf gD\M\Q}\}$.

\snoi {\rm (iv)}
If one of $\sigma_{\bf gK\M}(T)$, $\sigma_{\bf gD\M}(T)$, $\sigma_{\bf gD\M\W}(T)$, $\sigma_{\bf gD\M\Phi}(T)$, $\sigma_{\bf gD\M\W_+}(T)$, $\sigma_{\bf gD\M\Phi_+}(T)$, $\sigma_{\bf gD\M\J}(T)$, $\sigma_{\bf gD\M\W_-}(T)$, $\sigma_{\bf gD\M\Phi_-}(T)$, $\sigma_{\bf gD\M\Q}(T)$ is finite (countable), then all of them are equal and hence finite (countable).
\end{theorem}
\begin{proof} Since $\sigma_{\bf gK\M}(T)$ and $\sigma_{\bf gD\M R_i}(T)$, $1\le i\le 9$, are compact,
according to \eqref{spec.1} and the  inclusions

{\footnotesize
$$\begin{array}{ccccccccccc}
&&&&   \sigma_{\bf gD\M \Phi_+}(T) & \subset &    \sigma_{\bf gD\M \W_+}(T) &  \subset  & \sigma_{\bf gD\M \J}(T)& & \\
&&& \rotatebox{20}{$\subset$} & & \rotatebox{-20}{$\subset$}&
&\!\rotatebox{-20}{$\subset$}& & \rotatebox{-20}{$\subset$}&\\
  & &  \sigma_{\bf gK\M}(T) & \subset &
 & \subset &  \sigma_{\bf gD\M \Phi}(T)&
  \subset  & \sigma_{\bf gD\M \W}(T)&\subset&   \sigma_{\bf gD\M }(T).\\
&&& \rotatebox{-20}{$\subset$} & & \rotatebox{20}{$\subset$}&
&\rotatebox{20}{$\subset$}& & \rotatebox{20}{$\subset$}&\\
& &&&   \sigma_{\bf gD\M \Phi_-}(T)& \subset&  \sigma_{\bf gD\M \W_-}(T) & \subset & \sigma_{\bf gD\M \Q}(T) & &\\
\end{array}$$
}
it is enough to prove that
\begin{equation}\label{pot1}
  \partial\, \sigma_{{\bf g D\M R}_i}(T)\subset\sigma_{\bf gK\M}(T),\ 1\le i\le 9.
\end{equation}
Suppose that $\lambda\in \partial\sigma_{{\bf g D\M R}_i}(T)$.  From \eqref{glava} and Corollary \ref{cor1*} (iii) it follows that
\begin{equation}\label{pot2}
   \sigma_{{\bf g D\M R}_i}(T)= \sigma_{\bf gK\M}(T) \cup\inter\,  \sigma_{{\bf g D\M R}_i}(T).
\end{equation}
 Since $\sigma_{{\bf g D\M R}_i}(T)$ is closed, it follows that $\lambda\in \sigma_{{\bf g D\M R}_i}(T)$,  and 
 from \eqref{pot2} we conclude that $\lambda\in\sigma_{\bf gK\M}(T)$.
   \end{proof}

\begin{theorem}\label{Lak}
Let $T\in L(X)$.  The following statements are equivalent:

\snoi {\rm (i)} $\sigma_{\bf gK\M}(T)=\emptyset$;

\snoi {\rm (ii)} $\sigma_{\bf gD\M}(T)=\emptyset$;

\snoi {\rm (iii)} $T$ is polynomially meromorphic;

\snoi  {\rm (iv)} $\sigma_{\bf B\B}(T)$ is a finite set.

\end{theorem}
\begin{proof} The equivalence {\rm (i)}$\Longleftrightarrow${\rm (ii)}  follows from Theorem \ref{potrebna}.

The equivalence
{\rm (iii)}$\Longleftrightarrow${\rm (iv)}  has been proved in \cite{BDHZ}.

{\rm (ii)} $\Longrightarrow$ {\rm (iv)}:  Suppose that $\sigma_{\bf gD\M}(T)=\emptyset$. From Theorem \ref{glavna}, (iv) $\Longleftrightarrow$ (ix), it follows that $\sigma_{\bf g D\M}(T)=\sigma_{\bf g K\M}(T)\cup \acc\, \sigma_{\bf B\B}(T)$, and so $\acc\, \sigma_{\bf B\B}(T)=\emptyset$  which implies that $\sigma_{\bf B\B}(T)$ is a finite set.

{\rm (iii)} $\Longrightarrow$ {\rm (ii)}: Let $T$ be polynomially meromorphic and $p^{-1}(0)=\{\lambda_1,\dots,\lambda_n\}$ where $p(T)\in (\M)$. According to \cite[Corollary 4.3]{BDHZ} we have that $\sigma_{\bf B\B}(T)\subset p^{-1}(0)$. It implies that  $T-\lambda$ is Drazin invertible  and hence, generalized Drazin-meromorphic invertible     for every $\lambda\notin p^{-1}(0)$.
\ According to \cite[Theorem 4.9]{BDHZ}, $X$ is  decomposed into the
direct sum $X=X_1\oplus\dots\oplus X_n$ where $X_i$ is closed
$T$-invariant subspace of $X$,  $T=T_1\oplus\dots\oplus T_n$
where $T_i$ is the reduction of $T$ on $X_i$ and $T_i-\lambda_i\in (\M)$,
 $i=1,\dots, n$. From $T_i-\lambda_i\in (\M)$, it follows that $\sigma_{\bf B\B}(T_i-\lambda_i)\subset \{0\}$ and hence,  $\sigma_{\bf B\B}(T_i)\subset \{\lambda_i\}$, $i=1,\dots,n$. It implies that $T_i-\lambda_j$ is B-Browder for $i\ne j$, $i,j\in\{1,\dots,n\}$.

Consider  the  decomposition
$$T-\lambda_1=(T_1-\lambda_1)\oplus(T_2-\lambda_1)\oplus\dots\oplus (T_n-\lambda_1).
$$
From Lemma \ref{zatvorenost} it follows that $ X_2\oplus\dots\oplus X_n$ is closed.
Since $(X_1, X_2\oplus\dots\oplus X_n)\in Red(T)$, $(T-\lambda_1)_{X_1}=T_1-\lambda_1\in (\M)$, and since  $(T-\lambda_1)_{X_2\oplus\dots\oplus X_n}=(T_2-\lambda_1)\oplus\dots\oplus (T_n-\lambda_1)$
is B-Browder as a direct sum of B-Browder operators $T_2-\lambda_1,\dots , T_n-\lambda_1$,
  it follows that $T-\lambda_1$ is generalized Drazin-meromorphic invertible. In that way   we can  prove that $T-\lambda_i$ is generalized Drazin-meromorphic invertible for every $i\in\{1,\dots,n\}$. Therefore, $T-\lambda$ is  generalized Drazin-meromorphic invertible     for every $\lambda\in\CC$, and so $\sigma_{\bf gD\M}(T)=\emptyset$.
\end{proof}

P. Aiena and E. Rosas \cite[Theorem 2.10]{aienarosas} proved that if
  $T\in L(X)$ be an operator for which $\sigma_{ap}(T)=\partial\sigma (T)$ and every $\lambda\in \partial\sigma (T)$ is not isolated in $\sigma (T)$, then
$\sigma_{ap}(T)=\sigma_{Kt}(T)$, while Q.  Jiang  and H. Zhong \cite[Theorem 3.12]{Kinezi} improved this result by proving that under the same conditions it holds $\sigma_{ap}(T)=\sigma_{gK}(T)$.
In \cite[Theorem 3.14]{SZC} it was proved that $\sigma_{ap}(T)=\sigma_{gKR}(T)$. The next theorem improves these results.

\begin{theorem} \label{PO} For $T\in L(X)$ suppose that  $\sigma_{ap}(T)=\partial\sigma (T)$ and every $\lambda\in \partial\sigma (T)$ is not isolated in $\sigma (T)$. Then
$$\sigma_{\bf gK\M}(T)=\sigma_{\bf gD\M\Phi_+}(T)=\sigma_{\bf gD\M\W_+}(T)=\sigma_{\bf gD\M\J}(T)=\sigma_{ap}(T).
$$
\end{theorem}
\begin{proof} From the proof of \cite[Theorem 3.14]{SZC} we have $\sigma_{ap}(T)=\acc\, \sigma_{ap}(T)=\partial\, \sigma_{ap}(T)$.
According to
 the inclusion \eqref{presek2} it holds
\begin{equation}\label{presek20*}
  \partial \, \sigma_{ap}(T)\cap \acc\, \sigma_{\bf  B\B_+}(T)\subset \sigma_{\bf g K\M}(T).
\end{equation}
From \cite[Corollary 4.9 (i)]{ZB} we have that $\sigma_{ap}(T)=\sigma_{TUD}(T)$, and since $\sigma_{TUD}(T)\subset\sigma_{\bf B\B_+}(T)\subset\sigma_{ap}(T)$ \cite{Ber0}, it follows that $\sigma_{\bf B\B_+}(T)=\sigma_{ap}(T)$. Hence $\partial \, \sigma_{ap}(T)\cap \acc\, \sigma_{\bf  B\B_+}(T)= \sigma_{ap}(T)$, which together with \eqref{presek20*} gives $\sigma_{ap}(T)\subset \sigma_{\bf g K\M}(T)$. As $\sigma_{\bf g K\M}(T)\subset\sigma_{ap}(T)$, we get that $\sigma_{ap}(T)=\sigma_{\bf gK\M}(T)
$.
\end{proof}
\begin{theorem}\label{kl} Let $T\in L(X)$ be an operator for which $\sigma_{su}(T)=\partial\sigma (T)$ and every $\lambda\in \partial\sigma (T)$ is not isolated in $\sigma (T)$. Then
$$\sigma_{\bf gK\M}(T)=\sigma_{\bf gD\M\Phi_-}(T)=\sigma_{\bf gD\M\W_-}(T)=\sigma_{\bf gD\M\Q}(T)=\sigma_{su}(T).
$$
\end{theorem}
\begin{proof}
Follows from \cite[Corollary 4.9 (ii)]{ZB}
analogously to the proof of Theorem \ref{PO}.
\end{proof}
Using Theorems \ref{PO} and \ref{kl} we find the generalized Kato-meromorphic spectra and $gD\M  R_i$-spectra, $1\le i\le 9$, for some operators.
\begin{example}{\em For the $Ces\acute{a}ro\ operator$ $C_p$ defined on the classical Hardy space $H_p(\DDD)$, $\DDD$ the open unit disc and $1<p<\infty$, by
$$
 (C_pf)(\lambda)=\ds\frac 1{\lambda}\int_0^{\lambda}\ds\frac{f(\mu)}{1-\mu}\, d\mu,\ \, {\rm for\ all\ }f\in H_p(\DDD)\ {\rm and\ }\lambda\in\DDD,
 $$
it is known that its spectrum is  the closed disc $\Gamma_p$ centered at $p/2$ with radius $p/2$, $\sigma_{gKR}(C_p)=\sigma_{gK}(C_p)=\sigma_{Kt}(C_p)=\sigma_{ap}(C_p)=\partial \Gamma_p$ and also $\sigma_{\Phi}(C_p)=\partial \Gamma_p$ \cite{Mill}, \cite{aienarosas}, \cite{SZC}. From Theorem \ref{PO} it follows that
 $\sigma_{\bf gK\M}(C_p)=\sigma_{\bf gD\M\Phi_+}(C_p)=\sigma_{\bf gD\M\W_+}(C_p)=\sigma_{\bf gD\M\J}(C_p)=\sigma_{ap}(C_p)=\partial \Gamma_p$, and
since 
 $\inter\, \sigma_{\Phi}(C_p)=\inter\, \sigma_{\Phi_-}(C_p)=\emptyset$, according to \eqref{glava} we have that
 $
 \sigma_{\bf gD\M\Phi}(C_p)=\sigma_{\bf gD\M\Phi_-}(C_p)=\sigma_{\bf gK\M}(C_p)=\partial \Gamma_p.
 $
  From
 $\sigma(C_p)= \Gamma_p$ and  $\sigma_{ap}(C_p)=\partial \Gamma_p$ it follows that  and $\sigma_{su}(C_p)= \Gamma_p$ which together with $\sigma_{\Phi}(C_p)=\partial \Gamma_p$ implies that $\sigma_{\W_-}(C_p)=\sigma_{\W}(C_p)=\Gamma_p$.
     Again from \eqref{glava} we conclude that $\sigma_{\bf gD\M\W_-}(C_p)=\sigma_{\bf gD\M\W}(C_p)=\sigma_{\bf gD\M\Q}(C_p)=\sigma_{\bf gD\M}(C_p)=\Gamma_p$.
}
\end{example}

\begin{example} \rm For  each  $X\in \{  c_0(\NN), c(\NN), \ell_{\infty}(\NN),
\ell_p(\NN)\}$, $p\ge 1$, and the forward and backward unilateral shifts
$U$, $V\in L(X)$ there are equalities $\sigma(U)=\sigma(V)=\DD$, $\sigma_D(U)=\sigma_D(V)=\DD$,  $
\sigma_{ap}(U)=\sigma_{su}(V)=\partial\DD$,  $\sigma_{\Phi}(U)=\sigma_{\Phi}(V)=\partial\DD$, where $\DD=\{\lambda\in\CC:|\lambda|\le 1\}$ \cite[Theorem 4.2]{ZDH}.
As in the previous example, from Theorem \ref{PO} we conclude that $\sigma_{\bf gK\M}(U)=\sigma_{ap}(U)=\partial\DD$, and hence
$$
 \sigma_{\bf gD\M\Phi}(U)=\sigma_{\bf gD\M\Phi_-}(U)=\sigma_{\bf gD\M\Phi_+}(U)=\sigma_{\bf gD\M\W_+}(U)=\sigma_{\bf gD\M\J}(U)=\partial \DD.
 $$
while from \eqref{glava} we get
$\sigma_{\bf gD\M\W_-}(U)=\sigma_{\bf gD\M\W}(U)=\sigma_{\bf gD\M\Q}(U)=\sigma_{\bf gD\M}(U)=\DD$.

From Theorem  \ref{kl} it follows that
$$
\sigma_{\bf gK\M}(V)=\sigma_{\bf gD\M\Phi_-}(V)=\sigma_{\bf gD\M\W_-}(V)=\sigma_{\bf gD\M\Q}(V)=\sigma_{su}(V)=\partial\DD.
$$
   As $\sigma_{\Phi}(V)=\partial\DD$, from \eqref{glava} we get that  $\sigma_{\bf gD\M\Phi}(V)=\sigma_{\bf gD\M\Phi_+}(V)=\partial\DD$.
   From $\sigma_{\Phi}(V)=\partial\DD$, $\sigma_{ap}(V)=\DD$ and $\sigma_{su}(V)=\partial\DD$, we conclude that for   $|\lambda|<1$ it holds that $V-\lambda I$ is Fredholm with positive  index and so, $\{\lambda\in\CC:|\lambda|<1\}\subset\sigma_{\W_+}(V)\subset\sigma_{\W}(V)\subset\DD$, which implies that
   $\sigma_{\W_+}(V)=\sigma_{\W}(V)=\DD$. Now  again from \eqref{glava} it follows that $\sigma_{\bf gD\M\W_+}(V)=\sigma_{\bf gD\M\W}(V)=\DD$, and hence $\sigma_{\bf gD\M\J}(V)=\DD$ and  $\sigma_{\bf gD\M}(V)=\DD$.

 Every  non-invertible isometry $T$ has the property that $\sigma(T)=\DD$ and $\sigma_{ap}(T)= \partial\DD$ \cite[p. 187]{aienarosas},  and hence   $\sigma_{ap}(T)=\partial\sigma (T)$ and every $\lambda\in \partial\sigma (T)$ is not isolated in $\sigma (T)$. By Theorem \ref{PO} and \eqref{glava} it follows that $\sigma_{\bf gK\M}(T)=\sigma_{\bf gD\M\Phi_+}(T)=\sigma_{\bf gD\M\W_+}(T)=\sigma_{\bf gD\M\J}(T)=\partial\DD$ and $\sigma_{\bf gD\M\Q}(T)=\DD$.
\end{example}

\noindent
\author{Sne\v zana
\v C. \v Zivkovi\'c-Zlatanovi\'c}

\noindent{University of Ni\v s\\
Faculty of Sciences and Mathematics\\
P.O. Box 224, 18000 Ni\v s, Serbia}

\noindent {\it E-mail}: {\tt mladvlad@mts.rs}

\bigskip

\noindent
\author{Bhagwati P. Duggal}

\noindent{University of Ni\v s\\
Faculty of Sciences and Mathematics\\
P.O. Box 224, 18000 Ni\v s, Serbia}

\noindent {\it E-mail}: {\tt bpduggal@yahoo.co.uk}

\end{document}